\documentclass[11pt]{amsart}
\usepackage{amsmath, amssymb, amscd, mathrsfs, url, pinlabel,verbatim}
\usepackage[pagebackref]{hyperref}
\usepackage[margin=1.25in,marginparwidth=1in,centering,letterpaper,dvips]{geometry}
\usepackage{color,dcpic,latexsym,graphicx,epstopdf,comment}
\usepackage[all]{xy}
\usepackage[dvipsnames]{xcolor}
\usepackage{tikz,pgfplots}
\usepackage{enumitem,upquote,tabularx,textcomp}
\usepackage[all]{hypcap}
\usepackage[color=blue!20!white,textsize=tiny]{todonotes}

\title{Khovanov homology and the cinquefoil}

\author[John A. Baldwin]{John A. Baldwin}
\address{Department of Mathematics \\ Boston College}
\email{john.baldwin@bc.edu}

\author[Ying Hu]{Ying Hu}
\address{Department of Mathematics\\University of Nebraska Omaha}
\email{yinghu@unomaha.edu}

\author[Steven Sivek]{Steven Sivek}
\address{Department of Mathematics\\Imperial College London}
\email{s.sivek@imperial.ac.uk}

\thanks{JAB was supported by  NSF FRG Grant DMS-1952707 and NSF CAREER Grant DMS-1454865.}

\makeatletter
\newtheorem*{rep@theorem}{\rep@title}
\newcommand{\newreptheorem}[2]{%
\newenvironment{rep#1}[1]{%
 \def\rep@title{#2 \ref{##1}}%
 \begin{rep@theorem}}%
 {\end{rep@theorem}}}
\makeatother

\newtheorem {theorem}{Theorem}
\newreptheorem{theorem}{Theorem}
\newtheorem {lemma}[theorem]{Lemma}
\newtheorem {proposition}[theorem]{Proposition}

\numberwithin{equation}{section}
\numberwithin{theorem}{section}

\theoremstyle{definition}

\newtheorem{remark}[theorem]{Remark}
\newtheorem*{remark*}{Remark}

\newlist{pcases}{enumerate}{1}
\setlist[pcases]{
  label=\bf{Case~\arabic*:}\protect\thiscase.~,
  ref=\arabic*,
  align=left,
  labelsep=0pt,
  leftmargin=0pt,
  labelwidth=0pt,
  parsep=0pt
}
\newcommand{\case}[1][]{%
  \if\relax\detokenize{#1}\relax
    \def\thiscase{}%
  \else
    \def\thiscase{~#1}%
  \fi
  \item
}

\newcommand{\Z}{\mathbb{Z}}

\newcommand{\R}{\mathbb{R}}
\newcommand{\C}{\mathbb{C}}

\newcommand{\F}{\mathbb{F}}
\newcommand{\Q}{\mathbb{Q}}
\newcommand{\spc}{\operatorname{Spin}^c}
\newcommand{\spinc}{\mathfrak{s}}

\newcommand{\lsl}{\mathit{sl}}

\newcommand{\cT}{\mathcal{T}}

\newcommand\hf{\mathit{HF}}
\newcommand\hfs{\hf^{symp}}
\newcommand\cf{\mathit{CF}}
\newcommand\cfs{\cf^{symp}}
\newcommand\hfk{\mathit{HFK}}
\newcommand\hfkhat{\widehat{\hfk}}
\newcommand\pfh{\mathit{PFH}}

\newcommand\HFK{\widehat{\mathit{HFK}}}

\DeclareFontFamily{U}{mathx}{\hyphenchar\font45}
\DeclareFontShape{U}{mathx}{m}{n}{
      <5> <6> <7> <8> <9> <10>
      <10.95> <12> <14.4> <17.28> <20.74> <24.88>
      mathx10
      }{}
\DeclareSymbolFont{mathx}{U}{mathx}{m}{n}
\DeclareFontSubstitution{U}{mathx}{m}{n}
\DeclareMathAccent{\widecheck}{0}{mathx}{"71}
\newcommand{\HMto}{\widecheck{\mathit{HM}}}

\newcommand{\hfp}{\mathit{HF}^+}

\newcommand{\id}{\operatorname{id}}

\newcommand{\mirror}[1]{\overline{#1}}

\newcommand{\Kh}{\mathit{Kh}}
\newcommand{\lk}{\mathrm{lk}}
\newcommand{\AKh}{\mathit{AKh}}
\newcommand{\AHI}{\mathit{AHI}}
\newcommand{\Khr}{\overline{\Kh}}

\usetikzlibrary{calc,intersections}
\tikzset{every picture/.style=thick}
\tikzset{link/.style = { white, double = black, line width = 1.75pt, double distance = 1.25pt, looseness=1.75 }}
\tikzset{crossing/.style = {draw, circle, dotted, minimum size=0.5cm, inner sep=0, outer sep=0}}
\pgfplotsset{compat=1.12}

\begin{document}

\begin{abstract}
We prove that  Khovanov homology with coefficients in $\Z/2\Z$ detects the $(2,5)$ torus knot. Our proof makes use of a wide range of deep tools in Floer homology, Khovanov homology, and Khovanov homotopy. We combine   these tools with classical results on the dynamics of surface homeomorphisms to reduce the detection question to a problem about mutually braided unknots, which we then solve with computer assistance.  

\end{abstract}

\maketitle
\section{Introduction}
\label{sec:intro}

We prove that Khovanov homology detects the  torus knot $ T(2,5)$. In our convention, the $\delta$-grading on reduced Khovanov homology is defined by $\delta = q/2-h$, where $q$ and $h$ are the quantum and homological gradings, respectively. Our main result is the following:

\begin{theorem}\label{thm:main}
Suppose $K\subset S^3$ is a knot whose reduced Khovanov homology over $\Z/2\Z$ is 5-dimensional and is supported in a single positive $\delta$-grading. Then $K = T(2,5)$.
\end{theorem}


We remark, by contrast,  that the Jones polynomial does not detect $T(2,5)$---for example, it does not distinguish $T(2,5)$ from the knot $10_{132}$.

All previous knot detection results in Khovanov homology \cite{km-unknot,bs-trefoil,bdlls} used Floer-theoretic   invariants to deduce that the knots in question were fibered of genus at most 1, and hence belonged to a finite list. This crucial finiteness fails in  genus 2 and beyond.\footnote{For example, there are infinitely many genus-2, fibered, strongly quasipositive knots with the same Alexander polynomial as $T(2,5)$; see  \cite{misev2017families}.}  In particular, it is a long-standing open question whether knot Floer homology detects $T(2,5)$, and it was previously thought that a positive answer to this question was  necessary for a Khovanov detection result as in Theorem 1.1. Remarkably, we are  able to  prove Theorem \ref{thm:main} without fully resolving the knot Floer detection question.\footnote{Indeed, $T(2,5)$ is presently  the only knot or  link which Khovanov homology but not Floer homology is known to detect.}

One of the key innovations in this paper is our use of the relationship between Heegaard Floer homology and periodic Floer homology to gain information about the monodromies of genus-2 fibered knots with the same knot Floer homology as $T(2,5)$.  
We  then  use this to prove Theorem \ref{thm:main}, drawing on and combining in novel ways  some deep tools in Khovanov homology and Khovanov homotopy which had not  previously been brought to bear on these sorts of detection questions. In particular, the detection result in Theorem \ref{thm:main} is one of the first topological applications of Lipshitz--Sarkar's stable Khovanov homotopy type to date. 

We outline our  proof of Theorem \ref{thm:main} in detail below.

\subsection{Proof outline}

Suppose $K$ satisfies the hypotheses of Theorem \ref{thm:main}. We use Dowlin's spectral sequence from Khovanov homology to knot Floer homology \cite{dowlin} to conclude that $K$ is a genus-2,   fibered, strongly quasipositive knot, whose bigraded knot Floer homology satisfies \begin{equation}\label{eqn:hfk-eq}\hfkhat(K;\Q) \cong \hfkhat(T(2,5);\Q),\end{equation} as in \cite{bdlls}. Let $(S,h)$ be an abstract open book corresponding to the fibration of $K$.

It is enough to prove that $K$ is a torus knot. An Alexander polynomial argument shows that $K$ is not a satellite knot. So we may assume for a contradiction that $K$ is hyperbolic. Then the monodromy $h$ is freely isotopic to a pseudo-Anosov homeomorphism $\psi:S\to S$. 
 
The isomorphism \eqref{eqn:hfk-eq} enables us to calculate the $\textrm{spin}^c$ graded Heegaard Floer homology  \begin{equation}\label{eqn:calc}\hfp(S^3_0(K\#\overline K);\Q).\end{equation} 
It follows from the  work of  Lee--Taubes \cite{lee-taubes} and Kutluhan--Lee--Taubes \cite{klt1} that 
the Heegaard Floer homology of the mapping torus of an area-preserving surface diffeomorphism encodes the degree-1 part of the periodic Floer homology of the diffeomorphism, as defined by Hutchings \cite{hutchings-index,hutchings-sullivan}. We use this, together with our calculation of \eqref{eqn:calc} and Cotton-Clay's work on symplectic Floer homology \cite{cotton-clay}, to prove that  the homeomorphism $\psi$ has no fixed points.\footnote{See Remark \ref{rmk:monotone} for why we use $\hfp(S^3_0(K\#\overline K))$ rather than $\hfp(S^3_0(K))$ to conclude this.} With this, and the fact that homeomorphisms of the closed genus-2 surface commute up to isotopy with the hyperelliptic involution, we  argue that $K$ is the lift of the braid axis $A$ in the branched double cover of $S^3$ branched along  the closure $B$ of a pseudo-Anosov 5-braid representative of the unknot. In particular, $K$ is a doubly-periodic knot with unknotted quotient $A$.

We then use Stoffregen--Zhang's work \cite{stoffregen-zhang} relating the Khovanov spectrum of $K$ to the annular Khovanov spectrum of its unknotted quotient $A$ (with respect to the embedding of $A$ in the solid torus complement of a neighborhood of $B$)\footnote{Stoffregen--Zhang's result for doubly-periodic knots requires the use of $\Z/2\Z$ coefficients, which is ultimately why Theorem \ref{thm:main} is stated over $\Z/2\Z$ rather than, say, $\Q$.}, together with the $\mathfrak{sl}_2(\C)$-action on annular Khovanov homology defined by Grigsby--Licata--Wehrli \cite{grigsby-licata-wehrli}, to prove that the annular Khovanov homology of $A$ in its maximal nonzero annular grading is 1-dimensional. Combined with the spectral sequence from annular Khovanov homology to annular instanton homology due to Xie \cite{xie} and studied further by Xie-Zhang \cite{xie-zhang}, this implies that $A$ is braided with respect to $B$. In other words, $A$ and $B$ are mutually braided unknots.

Work of Morton \cite{morton} then implies that $B$ is the closure of a Stallings 5-braid. There are finitely many such braids, and we use Sage \cite{sagemath} to determine that among the pseudo-Anosov candidates that produce knots $K$ with the right Alexander polynomial, none of them are mutually braided with their braid axes.  This gives us the desired contradiction.

\subsection{Other applications} While the title of this paper is about Khovanov homology, we expect the ideas introduced here---especially those regarding the relationship between knot Floer homology and periodic Floer homology---to have wider applicability. To wit, shortly after we posted this article, Ni and Ni--Zhang used the ideas in \S\ref{sec:floer} to study monodromies of L-space knots \cite{ni-monodromy} and characterizing slopes for $T(2,5)$ \cite{ni-zhang-torus}. Moreover, recent work of Baldwin--Li--Sivek--Ye \cite{blsy} uses our ideas and results to answer an old question of Kronheimer--Mrowka  from their work on the Property P conjecture, on irreducible $SU(2)$-representations of fundamental groups of knot surgeries.  

The results of this paper also quickly imply the following.

\begin{theorem} \label{thm:elliptic-surgeries}
Let $K \subset S^3$ be a knot of genus 2 other than $T(\pm2,5)$.  Then no nontrivial Dehn surgery on $K$ can produce a lens space or a prism manifold.
\end{theorem}

\begin{proof}
Suppose  $K\neq T(\pm2,5)$ and that $r$-surgery on $K$ is a lens space or a prism manifold; we can assume that $r>0$ by replacing $K$ by its mirror if necessary.  Then $S^3_r(K)$ is a Heegaard Floer L-space \cite[Proposition~2.3]{osz-lens}, so $K$ is a genus-2 L-space knot.  Then it follows by combining \cite[Theorem~1.2]{osz-lens} and \cite[Corollary~9]{hedden-watson} that $\hfkhat(K) \cong \hfkhat(T(2,5))$, hence $K$ is doubly periodic by Theorem~\ref{thm:main2}.  Since $K$ is a periodic, non-torus knot, no nontrivial surgery on $K$ can be cyclic \cite{wang-zhou} or Seifert fibered with base $\mathbb{RP}^2$ \cite[Theorem~1.3]{miyazaki-motegi-periodic}, a contradiction.
\end{proof}

The lens space part of Theorem~\ref{thm:elliptic-surgeries} also follows from results of Baker in \cite{baker-smallgenus}, and the prism manifold part can  be proved in a different way using  results of Ni--Zhang in \cite{ni-zhang-finite}.

\subsection{Organization} We review some basic facts about pseudo-Anosov maps of surfaces in \S\ref{sec:pA}. In \S\ref{sec:floer}, we use periodic Floer homology to prove that a knot $K\neq T(2,5)$ with the same knot Floer homology as $T(2,5)$ is the lift of the braid axis in the branched double cover of a pseudo-Anosov 5-braid representative of the unknot. In \S\ref{sec:computations}, we prove the computer-assisted results involving Stallings braids alluded to above. We use these results to complete the proof of Theorem \ref{thm:main} in \S\ref{sec:proof}.

\subsection{Acknowledgements} We thank Ken Baker, Andy Cotton-Clay, Yi-Jen Lee, and Filip Misev for helpful correspondence, and Tye Lidman, Robert Lipshitz, Gage Martin, and Yi Ni for interesting conversations and questions related to this work.  We also thank the referee for a careful reading of this paper and helpful feedback.

\section{Pseudo-Anosov maps}\label{sec:pA} In this section, we provide a very brief review of some basic facts and terminology related to pseudo-Anosov homeomorphisms of surfaces, so that the reader who is less familiar with this subject will have an easier time following the arguments in the next section.

Suppose $h:\Sigma\to \Sigma$ is a homeomorphism of a compact, orientable surface $\Sigma$ with (possibly empty) boundary. The Nielsen--Thurston classification  \cite{thurston-diffeomorphisms} states that $h$ is freely isotopic\footnote{The word \emph{freely} here indicates  that the isotopy  need not restrict to the identity on $\partial \Sigma$.} to a homeomorphism $\phi$ which is either:
\begin{itemize}
\item \emph{periodic}, meaning that $\phi^n=\id$  for some positive integer $n$;
\item \emph{reducible}, meaning that there is a non-empty set $c = \{c_1,\dots,c_n\}$  of disjoint, essential, simple closed curves in $\Sigma$ such that $\{\phi(c_i)\}_{i=1}^n = c$; or
\item \emph{pseudo-Anosov}.
\end{itemize} 
There is an analogous classification for homeomorphisms of marked or punctured surfaces.  We clarify that in the definition of reducible, a simple closed curve is \emph{essential} if it neither bounds a disk with one or fewer markings/punctures nor cobounds an annulus with a boundary component of the surface. 
We will focus on the case in which $\phi$ is pseudo-Anosov, in which case it is neither periodic nor reducible.

If $\phi$ is  pseudo-Anosov, then there exists a transverse pair of singular measured foliations, $(\mathcal{F}_s,\mu_s)$ and $(\mathcal{F}_u,\mu_u)$ of $\Sigma$, called the \emph{stable} and \emph{unstable} foliations of $\phi$, such that \[\phi(\mathcal{F}_s,\mu_s)=(\mathcal{F}_s,\lambda^{-1}\mu_s) \textrm{ and }\phi(\mathcal{F}_u,\mu_u)=(\mathcal{F}_u,\lambda\mu_u)\] for some real number $\lambda>1$ called  the \emph{dilatation} of $\phi$ \cite{thurston-diffeomorphisms}. The singularities of $\mathcal{F}_s$ and $\mathcal{F}_u$ in the interior of $\Sigma$ are required to be ``$n$-pronged saddles" with $n\geq 3$, as shown in Figure \ref{fig:sing3}. Each foliation must have at least one singularity on every boundary component, and each boundary singularity must have a neighborhood of the form shown in Figure \ref{fig:sing3}. For marked or, equivalently, punctured surfaces, the foliations are allowed to have $1$-pronged singularities at the markings/punctures, as shown in the figure. Note that the singularities of $\mathcal{F}_s$ and $\mathcal{F}_u$ in the interior agree and have the same numbers of prongs with respect to both foliations. Since $\phi$ respects the leaf structures of these foliations, it must in particular permute the singularities with the same number of prongs.

\begin{figure}[!htbp]

\labellist 
\hair 2pt 
\tiny\pinlabel $x$ at 482 34

\endlabellist 

\begin{center}
\includegraphics[width=13.5cm]{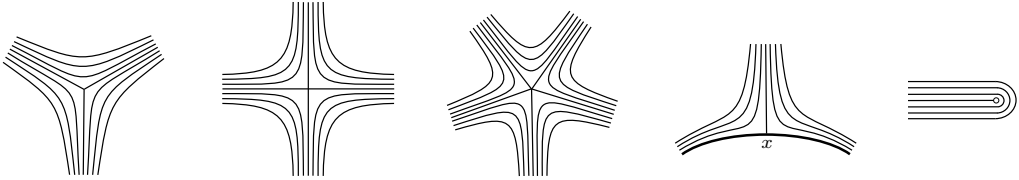}
\caption{ From left to right, one of the invariant foliations as seen in neighborhoods of: $n$-pronged interior singularities for $n= 3,4,5$; a singularity $x$ on $\partial \Sigma$ (the thickened segment represents a portion of $\partial \Sigma$); and  a 1-pronged marking/puncture.}
\label{fig:sing3}
\end{center}
\end{figure}

Suppose  $\Sigma$ has exactly one boundary component. If  $\mathcal{F}_s$ and $\mathcal{F}_u$ meet $\partial \Sigma$ in $n\geq 2$ prongs, then $\phi$ extends naturally to a pseudo-Anosov homeomorphism \[\hat\phi:\hat\Sigma\to\hat\Sigma\] of the closed surface $\hat\Sigma$ obtained from $\Sigma$ by capping off the boundary with a disk. Moreover, the invariant foliations of $\phi$ extend to stable and unstable foliations $\hat{\mathcal{F}}_s$ and $\hat{\mathcal{F}}_u$ for $\hat\phi$ in which the $n$-prongs on $\partial \Sigma$ extend to  $n$-prongs meeting at a singularity $p$ in the capping disk (except that $p$ is a smooth point when $n=2$). Note that $p$ is a fixed point of $\hat\phi$.

Suppose again that  $\Sigma$ has a single boundary component, and $h:\Sigma \to \Sigma$ is a homeomorphism which restricts to the identity  on  $\partial \Sigma$, such that $h$ is freely isotopic to a pseudo-Anosov map $\phi$.
The \emph{fractional Dehn twist coefficient}  of $h$, denoted by $c(h)$,  measures the amount of twisting about $\partial \Sigma$ in the free isotopy from $h$ to $\phi$. More precisely, let $x_1,\dots, x_n$ be the singularities of $\mathcal{F}_s$ on $\partial \Sigma$, labeled in the order  specified by the orientation of $\partial \Sigma$. There is an integer $m$ such that $\phi(x_i) = x_{i+m}$ for all $i$ (where the subscripts are taken mod $n$). If \[H:\Sigma\times [0,1]\rightarrow \Sigma\] is the free isotopy from $h$ to $\phi$, and \[G:\partial \Sigma\times [0,1]\rightarrow \partial \Sigma\times [0,1]\] is the map sending $(x,t)$ to $(H(x,t),t)$, then $G(x_i \times [0,1])$ is an arc from $(x_i,0)$ to $(x_{i+m},1)$.  Honda, Kazez, and Mati{\'c} define  \cite{hkm-veering}  \[c(h) := k/n\equiv m/n \textrm{ mod } \Z\]  to be the fractional number of times that $G(x_i\times [0,1])$ wraps around $\partial \Sigma\times [0,1]$. They  prove \cite[Theorem 1.1]{hkm-veering} that if the contact structure supported by the open book $(\Sigma,h)$ is tight, then $c(h)>0$ (in particular, $k>0$).

One can think of a braid $\beta$ as a homeomorphism of a marked disk which restricts to the identity on the boundary, and therefore define $c(\beta)$ in the same manner when the braid is pseudo-Anosov. Fractional Dehn twist coefficients of braids have been studied extensively by Ito and Kawamuro \cite{ito-kawamuro-obf}.

\section{Floer homology and fixed points}\label{sec:floer}

The  main goal of this section is to prove Theorem \ref{thm:main2} below, which will be used  in the proof of Theorem \ref{thm:main}, but may be of independent interest.

\begin{theorem}
\label{thm:main2}
Let $K\neq T(2,5)$ be a knot in $S^3$ such that \[\hfkhat(K;\Q)\cong \hfkhat(T(2,5);\Q)\] as bigraded vector spaces. Then there exists a pseudo-Anosov 5-braid $\beta$ whose closure $B=\hat \beta$ is an unknot with braid axis $A$, such that $K$ is the lift of  $A$ in the branched double cover \[\Sigma(S^3,B) \cong S^3.\]
In particular, $K$ is a doubly-periodic knot with unknotted quotient $A$ and axis $B$.
\end{theorem}

Theorem \ref{thm:main2} will follow from our main technical result, Theorem \ref{thm:fixed-point}, together with the fact that every pseudo-Anosov homeomorphism of the closed genus-2 surface is the lift of a pseudo-Anosov map of the  sphere with six marked points, under a branched double covering of the sphere branched along the marked points (Lemma \ref{lem:pA-cover}).

We will need a few preliminary results before proving Theorem \ref{thm:fixed-point}.

\begin{lemma}
\label{lem:hyperbolic}
Let $K\subset S^3$ be a  fibered knot  with Alexander polynomial \[\Delta_K(t) = t^2-t+1-t^{-1}+t^{-2}.\] Then either $K= T(2,\pm5)$ or $K$ is hyperbolic.
\end{lemma}

\begin{proof}
Since $K$ is fibered, the Alexander polynomial tells us that $K$ has genus 2. We know that $K$ is either a torus knot, a satellite knot, or hyperbolic. Since the only genus-2  torus knots are $T(2,\pm5)$, it suffices to prove that $K$ is not a satellite.

Suppose for a contradiction that $K = P(C)$ is a nontrivial satellite knot.  By ``nontrivial,'' we mean that the pattern $P \subset S^1\times D^2$ is not isotopic to the core $S^1\times \{0\}$, and the companion $C\subset S^3$ is not the unknot.  Since $K$ is fibered, the pattern $P$ has winding number $w \geq 1$, and both  $C$ and the satellite $P(U)$ are fibered \cite[Corollary~4.15, Proposition~5.5]{burde-zieschang}.  The Alexander polynomials of these knots are related by
\[ \Delta_K(t) = \Delta_{P(U)}(t) \cdot \Delta_C(t^w), \]
and $\Delta_C(t^w)$ is a nontrivial polynomial with degree $w\cdot g(C) \geq 1$.  Since $\Delta_K(t)$  is irreducible, we must then have $\Delta_{P(U)}(t)=1$.  But since $P(U)$ is also fibered, this can only happen if it has genus zero, meaning that  $P(U)$ is the unknot.

We now have $\Delta_K(t) = \Delta_C(t^w)$, which then forces  $w=1$.  Since $P$ has winding number one and $P(U)$ is the unknot, a result of Hirasawa, Murasugi, and Silver \cite[Corollary~1]{hms} says that $K=P(C)$ can only be fibered if $P$ is isotopic to the core $S^1\times\{0\} \subset S^1\times D^2$.  But this is a contradiction, so $K$ must not be a satellite after all.
\end{proof}

\begin{lemma}
\label{lem:prongs}
Let $K\subset S^3$ be a hyperbolic, fibered, strongly quasipositive knot with associated open book $(S,h)$. Then $h$ is freely isotopic to a pseudo-Anosov homeomorphism \[\psi:S \to S\] whose  stable foliation has $n\geq 2$ prongs on $\partial S$, and $h$ has fractional Dehn twist coefficient $c(h)=1/n$. 
\end{lemma}

\begin{proof} The complement $S^3\setminus K$ is given by the interior of the mapping torus of $h:S\to S$. Thurston proved \cite{thurston-fiber} that this manifold is hyperbolic iff $h$ is freely isotopic to a pseudo-Anosov homeomorphism \[\psi:S \to S. \] The  stable foliation of  $\psi$ meets the boundary of $S$ in   $n$ prongs for some integer $n\geq 1$. The fractional Dehn twist coefficient of $h$ is thus $c(h)=k/n$ for some integer $k$, as discussed in \S\ref{sec:pA}. Since $K$ is strongly quasipositive, the open book $(S,h)$ supports the tight contact structure on $S^3$ \cite{hedden-positivity,baader-ishikawa}, which implies \cite{hkm-veering} that   $k>0$, as mentioned in \S\ref{sec:pA}. 

On the other hand, Gabai and Oertel proved \cite[Theorem~5.3]{gabai-oertel} that the suspension of the stable foliation in the manifold supported by the open book $(S,h)$ is an essential lamination if $|k|>1$. (They stated this result in terms of degeneracy slopes rather than fractional Dehn twist coefficients, but ours is an equivalent formulation; see \cite[\S2.2]{ch}.) In the case at hand, this manifold is $S^3$, which does not contain essential laminations: for instance, it is shown in \cite{gabai-oertel} that a manifold with an essential lamination has universal cover $\R^3$. We therefore conclude that $c(h) = 1/n$.

It remains to show that $n\geq 2$. There are several ways to do this. For instance, if $n=1$, then $c(h) = 1$, in which case the manifold supported by $(S,h)$ has a taut foliation, as stated in \cite[Theorem 1.2]{hkm-rv2} which relies on \cite[Theorem 4.7]{roberts}. But this manifold is $S^3$, a contradiction.\end{proof}


\begin{lemma}\label{lem:hfk-fibered-sqp}
Let $K\subset S^3$ be a knot such that \[\hfkhat(K;\Q)\cong \hfkhat(T(2,5);\Q)\] as bigraded vector spaces. Then $K$ has genus 2, and is fibered and strongly quasipositive. 
\end{lemma}

\begin{proof}
The fact that \[\hfkhat(T(2,5),2;\Q)\cong \Q \textrm{ and }\hfkhat(T(2,5),i;\Q)=0 \textrm{ for } i>2\] implies that $K$ has genus 2 \cite{osz-genus} and is  fibered \cite{ni-hfk,juhasz-decomposition}.
The fact that the summand \[\hfkhat_0(T(2,5);\Q)\] in Maslov grading 0 is supported in Alexander grading $2$ implies that $\tau(K) = 2$, from which it follows that $K$ is strongly quasipositive, by \cite[Theorem 1.2]{hedden-positivity}.
\end{proof}

The following is our main technical result.

\begin{theorem}
\label{thm:fixed-point}
Let $K\subset S^3$ be a hyperbolic knot such that \[\hfkhat(K;\Q)\cong \hfkhat(T(2,5);\Q)\] as bigraded vector spaces. Let $(S,h)$ be an associated open book, and let \[\psi:S\to S\] be a pseudo-Anosov representative of $h$. Then $\psi$ has no fixed points.
\end{theorem}

\begin{proof}
By Lemma \ref{lem:hfk-fibered-sqp}, $K$ is a genus-2, fibered, strongly quasipositive knot. Then, by Lemma \ref{lem:prongs},  $h$ is  indeed freely isotopic to a pseudo-Anosov homeomorphism  $\psi$ whose stable foliation has $n\geq 2$ prongs on $\partial S$, and $c(h) = 1/n$. Since the latter is not an integer, $\psi$ permutes the boundary prongs nontrivially by rotating them.

To prove that $\psi$ has no fixed points, we first compute a portion of the Heegaard Floer homology of $S^3_0(K\#\overline{K})$, as discussed in the introduction. For the rest of the proof, we will work over $\Q$ without indicating these coefficients in the notation.

Note that $K\#\overline{K}$ is   a fibered knot, with open book given by the boundary connected sum \[(S,h) \#_b ({-}S,h) := (S\#_b {-}S, h\#_b h).\] The manifold $S^3_0(K\#\overline{K})$ can therefore be viewed as the mapping torus of the map \begin{equation}\label{eqn:g-map}g:F \to F,\end{equation} where $F$ is the closed genus-4 surface obtained by gluing $S$ to $-S$  along their boundaries, and $g$ is the map which restricts to $h$ on each copy of $S$. Let  $\spinc_k$ denote the $\spc$ structure on $S^3_0(K\#\overline{K})$ for which \[\langle c_1(\spinc_k),[F]\rangle = 2k.\] We will show below that \begin{equation}\label{eqn:zero-surgery}\hfp(S^3_0(K\#\overline{K}),\spinc_{-2})=0.\end{equation} 

By  conjugation symmetry, it suffices to prove that  \begin{equation}\label{eqn:zero-surgery2}\hfp(S^3_0(K\#\mirror{K}),\spinc_{2})=0.\end{equation}  We can compute this group using  Ozsv{\'a}th and Szab{\'o}'s result \cite[Theorem~1.4]{osz-alternating}, which provides a formula for  $\hfp(S^3_0(J),\spinc_{k})$ for any $k$ and any alternating  knot $J$.  Although their result is stated only for alternating knots and coefficients in $\Z$, the formula also holds for  any knot with the same bigraded knot Floer homology as an alternating knot (really, all that is needed is that the knot Floer homology is thin) and with coefficients in $\Q$. In particular, since $K$ has the same knot Floer homology as $T(2,5)$, it follows from the K{\"u}nneth formula that \[\hfkhat(K\#\overline{K})\cong \hfkhat(T(2,5)\#\overline{T(2,5)}).\] Since the latter connected sum is alternating, we may apply Ozsv{\'a}th and Szab{\'o}'s formula.

For this, first note that \[ \Delta_{K\#\mirror{K}}(t) = t^4 - 2t^3 + 3t^2 - 4t + 5 - 4t^{-1} + 3t^{-2} - 2t^{-3} + t^{-4}. \] The torsion coefficients $t_k$ of $K\#\mirror{K}$ are defined by \[t_k = \sum_{j=1}^\infty ja_{|k|+j},\] where $a_m$ is the coefficient of $t^m$ in the Alexander polynomial above. Observe that $t_{\pm 3} = 1$, and $t_k = 0$ for all other $k$. In the notation of \cite{osz-alternating},  $\sigma = 0$, and therefore \[\delta(\sigma,k)= \max\Big(0, \Big\lceil\frac{\sigma-2|k|}{4}\Big\rceil\Big) = 0.\] Then \cite[Theorem~1.4]{osz-alternating} says that for $k>0$, \[ \hfp(S^3_0(K\#\mirror{K}),\spinc_k) \cong \Q^{b_k} \oplus \left(\Q[U] / U^{\delta(\sigma,k)}\right) \]
where \[(-1)^{k+\sigma/2} b_k = \delta(\sigma,k) - t_k.\]  But $\delta(\sigma,2) = 0$ and hence $b_2 = \delta(\sigma,2)-t_2 = 0$ as well. This proves \eqref{eqn:zero-surgery2} and thus \eqref{eqn:zero-surgery}.

We  will next combine this computation with a series of difficult isomorphisms proved by others to show that the symplectic Floer homology of the map $g$ in \eqref{eqn:g-map} vanishes. 

Let $Y = S^3_0(K\#\mirror{K})$ for convenience of notation. Choose an area form $\omega$ on $F$ and let $\hat g$ be an area-preserving diffeomorphism in the mapping class of $g$. Let $\Gamma = [\mu]$ be the generator of   \[H_1(Y ;\Z)\cong \Z\] satisfying $\Gamma \cdot [F] = 1$, where $\mu$ is a meridian of $K\#\mirror{K}$, and therefore a section of the fibration $\pi: Y \to S^1$. Then $\Gamma$ is a positive monotone class with respect to $\omega$, in the sense of \cite[Definition~1.1]{lee-taubes}.  This means that if we define the class
\begin{equation} \label{eq:c-gamma}
c_\Gamma = 2 \mathit{PD}(\Gamma) + e(\ker d\pi)
\end{equation}
in $H^2(Y)$, then there is a relation of the form $[w_{\hat{g}}] = -\lambda c_\Gamma$, where $\lambda$ is a positive real number and $w_{\hat{g}}$ is a certain closed 2-form on $Y$ obtained from the $\hat{g}$-invariant area form $\omega$ on $F$.  Since $H^2(Y;\R)\cong\R$, both sides of \eqref{eq:c-gamma} are determined by their evaluation on $[F]$, so the positive monotone condition is further equivalent to $-\lambda\big(2(\Gamma\cdot [F]) + \chi(F)\big) > 0$, hence to $\chi(F) < -2(\Gamma\cdot[F]) = -2$; it is satisfied as long as $F$ has genus at least three.

There is an isomorphism \[ \pfh(\hat g,\Gamma) = \pfh_1(\hat{g}) \cong \hfs_\ast(\hat{g})  \]
relating Hutchings's periodic Floer homology \cite{hutchings-index,hutchings-sullivan} to symplectic Floer homology, as explained in \cite[Appendix~7.2]{lee-taubes}. We  will  show that the left-hand side vanishes, using our computation  \eqref{eqn:zero-surgery} together with the relationship between Heegaard Floer homology and periodic Floer homology, due to Lee--Taubes \cite{lee-taubes} and Kutluhan--Lee--Taubes \cite{klt1}.

Lee and Taubes \cite[Theorem~1.1]{lee-taubes} proved that
\[ \pfh(\hat{g}, \Gamma) \cong \HMto^{-*}(Y,\spinc_{-2},c_+), \]
where  $c_+$ denotes a positively monotone perturbation.  Since $\spinc_{-2}$ is non-torsion, Kronheimer and Mrowka proved that
\[ \HMto^{-\ast}(Y,\spinc_{-2},c_+) \cong \HMto^{-\ast}(Y,\spinc_{-2}) \cong \HMto^{-\ast}(Y,\spinc_{-2},c_b), \]
where $c_b$ denotes a balanced perturbation.  These isomorphisms are \cite[Theorem~31.1.2]{km-book} and \cite[Theorem~31.1.1]{km-book}, respectively: both isomorphisms are stated for the \emph{completed} version of monopole Floer homology, but there is no difference between that version and the uncompleted version above for non-torsion $\spc$ structures.  Finally, Kutluhan, Lee, and Taubes \cite[Main~Theorem]{klt1} proved that
\[ \HMto^{-\ast}(Y,\spinc_{-2},c_b) \cong \mathit{HF}^{-\ast}_+(Y,\spinc_{-2}). \]
We showed above  \eqref{eqn:zero-surgery} that $ \hfp(Y,\spinc_{-2}) = 0, $ so the Heegaard Floer cohomology vanishes as well,
\[\mathit{HF}^{-\ast}_+(Y,\spinc_{-2})=0,\]  by the Universal Coefficient Theorem. 
 Combining  these isomorphisms, we conclude that the symplectic Floer homology of $g$ vanishes, \begin{equation}\label{eqn:zero-g}\hfs_\ast(g):=\hfs_\ast(\hat g) = 0,\end{equation} as claimed. To deduce from this that $\psi$ has no fixed points, we apply Cotton-Clay's work on computing symplectic Floer homology \cite{cotton-clay}.
 

The map $g$ is reducible, with reducing curve $\partial S\subset F$ which separates $F$ into the components $\pm S$. The restriction of $g$ to each component is the monodromy $h$, which is isotopic to the pseudo-Anosov map $\psi$. Cotton-Clay proves \cite[Theorem 4.16]{cotton-clay}  that  $\hfs_\ast(g)$ has as a direct summand the symplectic Floer chain complex $\cfs_\ast(\psi_{sm})$ associated with a certain \emph{canonical smooth} representative $\psi_{sm}$ of $\psi$. Since $\psi$ rotates the boundary prongs  (there are no Type IIId fixed points, in the terminology of \cite[\S4]{cotton-clay}), the complex $\cfs_\ast(\psi_{sm})$ is   generated simply by the fixed points of $\psi_{sm}$, which are greater than or equal in number to the fixed points of $\psi$. Thus, \begin{equation}\label{eqn:dimfix}\dim \cfs_\ast(\psi_{sm}) \geq |\textrm{Fix}(\psi)|.\end{equation} The fact that $\hfs_\ast(g)$ vanishes \eqref{eqn:zero-g} then implies that \[\cfs_*(\psi_{sm}) = 0,\] from which we conclude that $\psi$ has no fixed points by \eqref{eqn:dimfix}, as desired.

We remark that the calculation of $\hfs_\ast(g)$ in \cite{cotton-clay} is stated with coefficients in $\Z/2\Z$, but everything we have used above also holds over $\Q$. Indeed, the fact that $\cfs_*(\psi_{sm})$ is a direct summand of $\hfs_\ast(g)$ comes from the facts that: 1) the latter is the homology of the symplectic Floer chain complex of a \emph{perturbed standard form} representative of $g$, of which the complex $\cfs_*(\psi_{sm})$ is a direct summand, and 2) the differential on $\cfs_*(\psi_{sm})$ vanishes. Both  hold for purely topological and/or index reasons, and are thus true  over any coefficient ring.
\end{proof}

\begin{remark}
\label{rmk:monotone} The reader may wonder why in the proof of Theorem \ref{thm:fixed-point} we use the Heegaard Floer homology of $S^3_0(K\#\mirror{K})$ rather than that of $S^3_0(K)$ to show that $\psi$ has no fixed points. The reason is that  the isomorphism of Lee--Taubes \cite{lee-taubes}, \[\pfh(\hat g, \Gamma)=\pfh_1(\hat{g}) \cong \hfs_\ast(\hat{g}),\]   used in the proof  requires that $\Gamma$ is monotone.  This would not hold if we replaced $K\#\mirror{K}$ with  $K$: in genus 2, the degree-1 class $\Gamma$ is not monotone. There is a similar isomorphism in the non-monotone case, using  local coefficients rather than $\Q$ \cite[Theorem~6.2]{lee-taubes}, and a relationship between this version of $\pfh$ and the Heegaard Floer homology  $S^3_0(K)$ in the torsion $\spc$ structure, with local coefficients, but it is not quite as clear that we can compute the latter from the knot Floer homology of $K$, and it is not as clear how to interpret the result of such a computation in terms of the fixed points of $\psi$. 
\end{remark}

We will use the lemma below in the proofs of Theorem \ref{thm:main2} and Proposition \ref{prop:prong-structure}.

\begin{lemma}
\label{lem:pA-cover}
Let $\phi:\Sigma\to\Sigma$ be a pseudo-Anosov map of a closed genus-2 surface $\Sigma$. Then there exists a branched double covering \[\pi:\Sigma \to S^2,\] branched along six   points $q_1,\dots,q_6\in S^2$, such that $\phi$ is a lift of a pseudo-Anosov map \[b:(S^2,q_1,\dots,q_6) \to (S^2,q_1,\dots,q_6)\] of the marked sphere, and the invariant  foliations of $\phi$ are lifts of those of $b$.
\end{lemma}

\begin{proof}
Let $\pi:\Sigma \to S^2$ be a branched double covering, branched along six points $q_1,\dots,q_6\in S^2$. Let $p_1,\dots,p_n$ be the preimages of these branch points, \[p_i:=\pi^{-1}(q_i).\] Let $\tau$ denote the covering involution of this map, also known as the hyperelliptic involution of $\Sigma$. It is well-known that the mapping class group of the closed genus-2 surface $\Sigma$ has nontrivial center, generated by the isotopy class of $\tau$ (see \cite[\S 3.4]{farb-margalit}, for example).\footnote{This is special to genus ($\leq$) $2$, as the mapping class groups of higher genus surfaces have trivial centers.} That is, every homeomorphism of $\Sigma$ commutes with $\tau$ up to isotopy.  In fact, the mapping class group of $\Sigma$ is generated by Dehn twists about simple closed curves which are invariant under $\tau$ \cite[Theorem 4.1]{farb-margalit}. Since each  Dehn twist in this generating set can be represented by a homeomorphism which commutes with $\tau$ (on the nose), it follows  that every homeomorphism of $\Sigma$ is isotopic to one which commutes with $\tau$ (on the nose).\footnote{This is \emph{a priori} stronger than the  statement that every homeomorphism commutes with $\tau$ up to isotopy.} In particular,  $\phi$ is isotopic to a map $\phi_0$ satisfying $\phi_0\tau=\tau\phi_0$. This commutativity implies that $\phi_0$ is in fact a map \[\phi_0:(\Sigma,p_1,\dots,p_6) \to (\Sigma,p_1,\dots,p_6)\] of the marked genus-2 surface, and is a lift of a map \[b_0:(S^2,q_1,\dots,q_6) \to (S^2,q_1,\dots,q_6)\] of the marked sphere, under the covering map $\pi$. 
We show below that since $\phi_0$ is isotopic to a pseudo-Anosov map, so is the restriction of $b_0$ to the punctured sphere $S^2\setminus \{q_1,\dots,q_6\}$.

Suppose this restriction is isotopic to a periodic map with period $n$. Then $\phi_0^n$ is isotopic to a lift of the identity map; hence, $\phi_0^{2n}$ is isotopic to the identity, a contradiction. 

Suppose instead that this restriction is isotopic to a reducible  map which fixes a collection $\{c_1,\dots,c_m\}$ of disjoint, essential, simple closed curves  in the punctured sphere. Then $\phi_0$ fixes, up to isotopy, the collection of simple closed curves arising as the components of the preimages $\pi^{-1}(c_1),\dots,\pi^{-1}(c_m)$. We claim  that each  component of $\pi^{-1}(c_i)$ is essential in $\Sigma$, for every $i$. Otherwise, there is a component $c$ of some $\pi^{-1}(c_i)$ which bounds a disk $D\subset \Sigma$. We may assume without loss of generality that $c$ is innermost in $D$ among components of $\pi^{-1}(c_i)$. If $D$ does not contain any $p_i$ then $c_i = \pi(c)$ bounds the disk $\pi(D)$ in the punctured sphere, which implies that $c_i$ is not essential, a contradiction. Suppose then that $D$ contains some $p_i$. Since $\tau(c)$ is either equal to or disjoint from $c$, and $c$ is innermost, either: \[ D\cap \tau(D) = \emptyset\, \textrm{ or } \,D\subsetneq \tau(D)\, \textrm{ or } \,D = \tau(D).\] We can rule out the first case  since $\tau(p_i)=p_i$. We can also rule out the second case, since by applying $\tau$ to both sides, we would have proper containment in the other direction, \[\tau(D)\subsetneq D, \] and hence $D \subsetneq D$, contradiction. In the third case, the disk $D$  is symmetric with respect to $\tau$.
The map  $\pi$ therefore restricts to a branched double covering  \[\pi:D\to \pi(D).\] It follows that $D$ contains only one $p_i$, since a disk is not the branched double cover of a disk along more than one point. Then $c_i=\pi(c)$ bounds the    once-punctured disk \[\pi(D)\cap (S^2\setminus\{q_1,\dots,q_6\}),\] which implies that $c_i$ is not essential in the punctured sphere, a contradiction. We conclude that all components of every multicurve $\pi^{-1}(c_i)$ are essential in $\Sigma$. Therefore,  $\phi_0$ fixes, up to isotopy, a collection of disjoint, essential, simple closed curves in $\Sigma$. That is,  $\phi_0$ is isotopic to a reducible map, a contradiction.

We have  shown that the restriction of $b_0$ to  $S^2\setminus \{q_1,\dots,q_6\}$ is isotopic to a pseudo-Anosov map. Let \[b:(S^2,q_1,\dots,q_6) \to (S^2,q_1,\dots,q_6)\] denote the extension of this pseudo-Anosov  map to the closed, marked sphere. Then $b$ is a pseudo-Anosov map isotopic to $b_0$. Let \[f_t:(S^2,q_1,\dots,q_6) \to (S^2,q_1,\dots,q_6)\] be an isotopy with $f_0=b_0$ and $f_1 = b$. There is a unique lift \[\tilde f_t:\Sigma \to \Sigma\] with respect to $\pi$ for which $\tilde f_0 = \phi_0$. By construction, $\tilde f_1$ is a pseudo-Anosov map, isotopic to $\phi_0$, whose invariant foliations are lifts of the invariant foliations of $b$. By the uniqueness of pseudo-Anosov representatives of isotopy classes \cite[Theorem 12.5]{flp}, there exists a diffeomorphism $g:\Sigma\to\Sigma$, isotopic to the identity, such that \[\phi = g\tilde f_1g^{-1}.\] If we relabel, letting $\pi$ denote the map $\pi g^{-1}$, then $\phi$ is a lift of $b$ under the map \[\pi:\Sigma \to S^2,\] such that the invariant  foliations of $\phi$ are lifts of those of $b$, as desired.
\end{proof}

\begin{proof}[Proof of Theorem \ref{thm:main2}] Suppose $K$ satisfies the hypotheses of the theorem. Then, by  Lemma \ref{lem:hfk-fibered-sqp}, $K$ is a genus-2, fibered, strongly quasipositive knot. By Lemma \ref{lem:hyperbolic}, $K$ is hyperbolic. By Lemma \ref{lem:prongs},  $h$ is then freely isotopic to a pseudo-Anosov map  $\psi$  whose  invariant foliations  have $n\geq 2$ prongs on $\partial S$.

Since the invariant  foliations of $\psi$  have more than one boundary prong,  $\psi$  extends to a pseudo-Anosov homeomorphism \[\hat\psi:\hat S\to \hat S\] of the closed genus-2 surface $\hat S$ obtained from $S$ by capping off its boundary with a disk, as discussed in \S\ref{sec:pA}. The invariant foliations for $\psi$ extend to invariant foliations for $\hat\psi$ in which the $n$ boundary prongs extend to an $n$-pronged singularity (or smooth point if $n=2$) $p$ in the disk, which is fixed by $\hat\psi$. It follows from Theorem \ref{thm:fixed-point} that $p$ is the only fixed point of $\hat\psi$.

By Lemma \ref{lem:pA-cover}, there exists a branched double covering \[\pi:\hat S\to S^2\] of the sphere along six  points $q_1,\dots,q_6$ such that $\hat\psi$ is the lift of a pseudo-Anosov map \[b:(S^2,q_1,\dots,q_6) \to (S^2,q_1,\dots,q_6)\] of the marked sphere. Let \[\tau:\hat S\to \hat S\] be the associated covering involution, and note that $\tau\circ \hat\psi = \hat\psi\circ \tau.$  Moreover, the fixed points of $\tau$ are precisely the preimages $p_1,\dots, p_6$, where  \[p_i:=\pi^{-1}(q_i).\] 
 We claim that the fixed point $p$ of $\hat\psi$ is one of these $p_i$; that is, \begin{equation}\label{eqn:taup}\tau(p)=p.\end{equation} To see this, we note that \[\hat\psi(\tau(p)) = \tau(\hat\psi(p)) = \tau(p).\] That is, $\tau(p)$ is also a fixed point of $\hat\psi$. Since $\hat\psi$ has only one fixed point,  \eqref{eqn:taup} follows.

Without loss of generality, let  us suppose  $p=p_6$. Then $\pi$ restricts to a branched double covering of  punctured surfaces \[\pi':\hat S\setminus \{p_6\} \to S^2 \setminus \{q_6\}.\] We will  view these punctured surfaces as the interiors of $S$ and $D^2$. Let us then extend $\pi'$ to a branched double covering between compact surfaces,  \[\pi':S \to D^2,\] branched  along the five marked points $q_1,\dots,q_5$. The extension of \[\hat\psi|_{\hat S\setminus \{p_6\}\cong \textrm{int}(S)}\] to $S$ is freely isotopic to $h$. It follows that $h$ is isotopic to the lift under $\pi'$ of a homeomorphism \[\beta:(D^2,q_1,\dots,q_5) \to (D^2,q_1,\dots,q_5) \] of the marked disk which is the identity on $\partial D^2$, where $\beta$ is freely isotopic to the extension of \[b|_{ S^2\setminus \{q_6\}\cong \textrm{int}(D^2)}\] to $D^2$.
In what follows, we will think of  $\beta$ as a homeomorphism of this marked disk and as a 5-braid, interchangeably.

This map specifies an open book decomposition $(D^2,\beta)$ of $S^3$ with unknotted binding $A$. In this open book decomposition, the points $q_1,\dots,q_5$ sweep out the closure \[B = \hat\beta\subset S^3\] of the 5-braid $\beta$, with axis $A$. 

The covering map $\pi'$ extends to a branched double covering from the open book decomposition specified by $(S,h)$ to the open book decomposition specified by $(D^2,\beta)$, in which the branch set is the braid closure $B$. Precisely, this extension is defined by \begin{equation}\label{eqn:bdc}\pi'\times \id: M_{h} \to M_{\beta},\end{equation}  where $M_h\cong S^3$ is the manifold associated to the open book $(S,h)$, given by \[M_{h}: = (S\times[0,1])/\sim,\] where $\sim$ is the relation defined by \begin{align*}
(x,0)\sim (h(x),1)& \textrm{ for  }x\in S\\
(x,t) \sim (x,s)& \textrm{ for  }x\in \partial S \textrm { and }s,t\in[0,1],
\end{align*} 
and likewise for $M_\beta \cong S^3$.
 Since  \[M_{h}\cong \Sigma(S^3,B)\cong S^3,\]  it follows that $B$ is an unknot \cite{waldhausen-involution}. Finally, the  binding \[K = (\partial S\times\{0\})/\sim\] of $(S,h)$ is the lift of the binding  \[A = (\partial D^2\times\{0\})/\sim\] of $(D^2,\beta)$ (and braid axis of $\beta$) under the branched double covering \eqref{eqn:bdc}, completing the proof of Theorem \ref{thm:main2}.
\end{proof}

Although we will not need it for the proof of Theorem \ref{thm:main}, we can say a bit more about the monodromy of knots  which satisfy (something slightly weaker than) the hypotheses of Theorem \ref{thm:fixed-point}:

\begin{proposition}
\label{prop:prong-structure}
Suppose $K\subset S^3$ is a hyperbolic, fibered, strongly quasipositive knot with Alexander polynomial \[\Delta_K(t) = t^2-t+1-t^{-1}+t^{-2}.\] Let $(S,h)$ be an associated open book, and let \[\psi:S\to S\] be a pseudo-Anosov representative of $h$. If $\psi$ has no  fixed points then the stable foliation of $\psi$ has 4 prongs on $\partial S$, and two interior 3-pronged singularities which are swapped by $\psi$. Furthermore, $h$ has fractional Dehn twist coefficient $c(h)=1/4$.
\end{proposition}

\begin{proof}
Since $K$ is fibered, the Alexander polynomial tells us that $K$ has genus 2. By Lemma \ref{lem:prongs}, the  stable foliation of $\psi$ has $n\geq 2$ prongs, and $c(h)=1/n$. As in the proof of Theorem \ref{thm:main2}, we  conclude that  $\psi$  extends to a pseudo-Anosov homeomorphism \[\hat\psi:\hat S\to \hat S\] of the closed genus-2 surface  obtained  by capping off  $\partial S$ with a disk, such that the invariant foliations for $\psi$ extend to  invariant foliations for $\hat\psi$ with an $n$-pronged singularity (or smooth point if $n=2$) $p$  in the disk, which is fixed by $\hat\psi$. Since $\psi$ has no fixed points, $p$ is the only fixed point of $\hat\psi$.

By Lemma \ref{lem:pA-cover}, there exists a branched double covering \[\pi:\hat S\to S^2\] of the sphere along six  points $q_1,\dots,q_6$ such that $\hat\psi$ is the lift of a pseudo-Anosov map \[b:(S^2,q_1,\dots,q_6) \to (S^2,q_1,\dots,q_6)\] of the marked sphere whose invariant foliations lift to those of $\hat\psi$. As in the proof of Theorem \ref{thm:main2},  we know that $\pi(p)$ is one of the branch points. It follows that the number $n$ of prongs at $p$ must be even.

As in the proof of Theorem \ref{thm:main2}, there is a 5-braid $\beta$ with unknotted closure $B = \hat\beta$ such that $K$ is the lift of the braid axis in the branched double cover $\Sigma(S^3,B) \cong S^3$. Since the open book $(S,h)$ associated to $K$ has fractional Dehn twist coefficient $c(h) = 1/n$, the braid $\beta$ has fractional Dehn twist coefficient $c(\beta) = 2/n$. Since the closure of $\beta$ is an unknot, we have that \[ |c(\beta)| < 1, \] by work of Ito and Kawamuro
\cite[Corollary~7.13(2a)]{ito-kawamuro-obf}. It follows that $n>2$. In particular, $p$ is a singular point of the invariant foliations of $\hat\psi$.

Let  $p = s_1, \dots s_k$ be the singular points of the invariant foliations of $\hat\psi$.  Suppose $s_i$ is an $a_i$-pronged singularity,  where $a_i \geq 3$ for each $i$. Then
\[ \sum_{i=1}^k (2-a_i) = 2\chi(\hat S) = -4 \]
by the Euler--Poincar\'e formula \cite[Proposition~5.1]{flp}. Since $a_1=n$ is an even integer greater than $2$, it follows that, up to permutations of $a_2,\dots,a_k$, the tuple $(a_1,\dots,a_k)$ must be one of: \begin{equation}\label{eqn:possible-prongs}(6) \textrm{ or } (4,4) \textrm{ or } (4,3,3).\end{equation}
In the first case, Masur and Smillie proved \cite[Theorem~2]{masur-smillie} that the invariant foliations of $\hat\psi$ must be orientable. The dilatation $\lambda>1$ of $\hat\psi$ is then the largest absolute value of a real eigenvalue of the map \[\hat\psi_*:H_1(\hat S) \to H_1(\hat S),\] by Thurston 
\cite{thurston-diffeomorphisms} (see \cite[Theorem~2.2]{lanneau-thiffeault}). Equivalently, $\lambda$ is the largest absolute value of a real root of the characteristic polynomial of $\psi_* = h_*$, which is the Alexander polynomial of $K$, \[\Delta_K(t) = t^2-t+1-t^{-1}+t^{-2}.\] But this is a cyclotomic polynomial, so it does not have any real roots, and we obtain a contradiction. This rules out the possibility $(6)$ in \eqref{eqn:possible-prongs}.

In the second case, since $\hat\psi$ fixes $p=s_1$ it must also fix  $s_2$. But then  $s_2$ would also  be a fixed point of $\psi$, which is a contradiction since we are assuming that $\psi$ has no fixed points. This rules out the possibility $(4,4)$ in \eqref{eqn:possible-prongs}.

We conclude that the third case, $(4,3,3)$, must reflect the prong structure for the invariant foliations of  $\hat\psi$. In particular, $n=4$, so that $c(h)=1/4$, and  the stable  foliation of $\psi$ has two 3-pronged singularities. The map $\psi$ must either swap these singularities or fix them. Since $\psi$ has no fixed points, we conclude that $\psi$ swaps them. This completes the proof.
\end{proof}

\begin{remark}It is  worth comparing Proposition \ref{prop:prong-structure} with Misev's work in  \cite{misev2017families}, in which he exhibits infinitely many genus-2, fibered, strongly quasipositive knots  with the same Alexander polynomial as $T(2,5)$ and pseudo-Anosov monodromies. The corresponding  invariant  foliations have the  prong structure    described in Proposition \ref{prop:prong-structure}, but the  3-pronged singularities are fixed rather than swapped, and the monodromies  have  regular fixed points.
\end{remark}

\begin{remark}
One  possible approach to proving that knot Floer homology detects $T(2,5)$ is to show that there are no  pseudo-Anosov maps $\psi$ satisfying the hypotheses and conclusions of Proposition \ref{prop:prong-structure}. Or, if such maps  exist, there are finitely many of them, whose knot Floer homology groups we then check by hand (or computer).
\end{remark}

\section{Exchangeable braids and computations}\label{sec:computations}

The  goal of this section is to  describe a computer-assisted proof of Theorem \ref{thm:fakeT25} below, which will  be used in the proof of Theorem \ref{thm:main} in the next section. 

A braid $\beta$ is said to be \emph{exchangeable} if its closure $B=\hat\beta$ is unknotted and the axis $A$ of $\beta$ is also braided with respect to $B$; in other words, if $A$ and $B$ are mutually braided unknots. If the braid index of $\beta$ is odd, then the braid axis $A$ lifts to a knot $K$ in the branched double cover $\Sigma(S^3,B)\cong S^3$.

\begin{theorem} \label{thm:fakeT25}
Let $\beta$ be a pseudo-Anosov 5-braid with unknotted closure.  Let $K$ be the lift of the braid axis in the branched double cover $\Sigma(S^3,\hat\beta)\cong S^3$.  If $K$ is strongly quasipositive with Alexander polynomial \[\Delta_K(t) = t^2-t+1-t^{-1}+t^{-2},\] then $\beta$ is not exchangeable.
\end{theorem}

We begin with the following preliminary lemmas.

\begin{lemma} \label{lem:stallings-form}
Let $\beta$ be an exchangeable $n$-braid for some odd $n$. Suppose  that the lift of the braid axis in the branched double cover $\Sigma(S^3,\hat\beta)\cong S^3$ is strongly quasipositive. Then $\beta$ is conjugate to a braid of the form
\[ \beta' = \sigma_{i_1j_1} \cdot \sigma_{i_2j_2} \cdot \ldots \cdot \sigma_{i_{n-1}j_{n-1}}, \]
where
\[ \sigma_{ij} = (\sigma_{j-1}\sigma_{j-2}\cdots\sigma_{i+1}) \cdot \sigma_i \cdot (\sigma_{j-1}\sigma_{j-2}\cdots\sigma_{i+1})^{-1} \]
for $1 \leq i < j \leq n$.
\end{lemma}

\begin{proof}
Morton \cite[Corollary~1.1]{morton} proved that every exchangeable $n$-braid is conjugate to a \emph{Stallings $n$-braid}, which is a product of the form
\[ \beta' = \sigma_{i_1j_1}^{\epsilon_1} \cdot \sigma_{i_2j_2}^{\epsilon_2} \cdot \ldots \cdot \sigma_{i_{n-1}j_{n-1}}^{\epsilon_{n-1}} \]
for some exponents $\epsilon_k = \pm 1$.   Then the closure $\hat\beta' \cong \hat\beta$ is smoothly unknotted, and represents a transverse knot $\cT$ with self-linking number $\lsl(\cT) = \left(\sum_k \epsilon_k\right) - n$.

Let $K$ be the lift of the braid axis of $\beta$ in the branched double cover $\Sigma(S^3,\hat\beta)\cong S^3$. Note that $K$ is a knot since $n$ is odd. Moreover, $K$   is the binding of an open book supporting  the contact manifold $(S^3,\xi)$ obtained as the branched double cover of the tight  $S^3$ branched along the transverse unknot $\cT$. Note that $\xi$ must be tight since $K$ is strongly quasipositive \cite{hedden-positivity,baader-ishikawa}. This implies that $\cT$ cannot be a stabilization of a transverse unknot \cite[Proposition~1.3]{plamenevskaya-branched}.  It follows that $\lsl(\cT)=-1$, and so all of the exponents $\epsilon_k$ must be $+1$.
\end{proof}

The next lemma appears as a remark in \cite[p.\ 88]{morton}, though the underlying construction goes back to Goldsmith \cite{goldsmith}.

\begin{lemma} \label{lem:fibered-lift}
If $\beta$ is an exchangeable $n$-braid, then the closure of $\beta^k$ is a fibered link for all integers $k \geq 1$.
\end{lemma}

\begin{proof}
Let $B = \hat\beta$ be the closure of $\beta$, with braid axis $A$.  Since $A$ is also braided about $B$, the complement of $B$ in $S^3$ is fibered by disks which each meet $A$ transversely at $n$ points.  If $\Sigma_k(S^3,A) \cong S^3$ is the branched $k$-fold cyclic cover of $A$, then these disks lift to surfaces which fiber the complement of the lift $\tilde{B}$.  But since $A$ is the braid axis of $\hat\beta$, the lifted link $\tilde{B} \subset \Sigma_k(S^3,A)$ is simply the closure of the braid $\beta^k$, so $\widehat{\beta^k}$ is fibered, as claimed.
\end{proof}

\begin{proof}[Proof of Theorem~\ref{thm:fakeT25}]
We apply Lemma~\ref{lem:stallings-form} in the case $n=5$. There are ten generators $\sigma_{ij}$ with $1 \leq i < j \leq 5$. Up to conjugacy, $\beta$ is a product of four such generators, according to Lemma \ref{lem:stallings-form}, so there are $10,000$ braids to check (in fact, many of these are conjugate to one another, so it is not strictly necessary to check all of  them).

We examine each of these 10,000 braids using Sage \cite{sagemath}, which can determine whether any one of these braids, such as
\[ \sigma_{13}\sigma_{24}\sigma_{14}\sigma_{25} = \sigma_2\sigma_1\sigma_2^{-1}\sigma_3\sigma_2^2\sigma_1\sigma_2^{-1}\sigma_3^{-1}\sigma_4\sigma_3\sigma_2\sigma_3^{-1}\sigma_4^{-1}, \]
is pseudo-Anosov:
\begin{verbatim}
sage: BG = BraidGroup(5)
sage: b = BG([2, 1, -2, 3, 2, 2, 1, -2, -3, 4, 3, 2, -3, -4])
sage: b.is_pseudoanosov()
True
\end{verbatim}

Since $K$ is the lift of the braid axis in the branched double cover $\Sigma(S^3,\hat\beta)$, its Alexander polynomial $\Delta_{K}(t)$ is the characteristic polynomial of the reduced Burau representation at $t=-1$, which has the form
\[ \rho: B_5 \to GL_4(\Z[t,t^{-1}]) \xrightarrow{t\mapsto -1} GL_4(\Z). \]
Indeed, let us view $\beta$ as a homeomorphism of the disk with five marked points. The branched double cover of this disk branched along the marked points is a genus-2 surface $S$ with one boundary component, and $\beta$ lifts to a homeomorphism \[h:S\to S\] such that $(S,h)$ is an open book corresponding to the fibered knot $K$. We know that $\Delta_{{K}}(t)$ is the characteristic polynomial of the monodromy action \[h_*:H_1(S) \to H_1(S),\] and $\rho$ computes precisely this action \cite[Proposition~2.1]{gambaudo-ghys}.  We can therefore use Sage to determine the braids $\beta$ for which  $\Delta_K(t) = t^2-t+1-t^{-1}+t^{-2}$  by:
\begin{verbatim}
sage: b.burau_matrix(reduced=True).subs(t=-1).characteristic_polynomial()
x^4 - x^3 + x^2 - x + 1
\end{verbatim}

There are exactly 20 braids out of the possible $10,000$ which are pseudo-Anosov and for which the lifted braid axis has this Alexander polynomial.  In the notation of Lemma~\ref{lem:stallings-form}, these are:
\[ \begin{array}{cccc}
\sigma_{13}\sigma_{24}\sigma_{14}\sigma_{25}, &
\sigma_{13}\sigma_{24}\sigma_{35}\sigma_{14}, &
\sigma_{13}\sigma_{24}\sigma_{35}\sigma_{25}, &
\sigma_{13}\sigma_{35}\sigma_{14}\sigma_{25}, \\

\sigma_{14}\sigma_{13}\sigma_{24}\sigma_{35}, &
\sigma_{14}\sigma_{25}\sigma_{13}\sigma_{24}, &
\sigma_{14}\sigma_{25}\sigma_{13}\sigma_{35}, &
\sigma_{14}\sigma_{25}\sigma_{24}\sigma_{35}, \\

\sigma_{24}\sigma_{14}\sigma_{25}\sigma_{13}, &
\sigma_{24}\sigma_{35}\sigma_{14}\sigma_{13}, &
\sigma_{24}\sigma_{35}\sigma_{14}\sigma_{25}, &
\sigma_{24}\sigma_{35}\sigma_{25}\sigma_{13}, \\

\sigma_{25}\sigma_{13}\sigma_{24}\sigma_{14}, &
\sigma_{25}\sigma_{13}\sigma_{24}\sigma_{35}, &
\sigma_{25}\sigma_{13}\sigma_{35}\sigma_{14}, &
\sigma_{25}\sigma_{24}\sigma_{35}\sigma_{14}, \\

\sigma_{35}\sigma_{14}\sigma_{13}\sigma_{24}, &
\sigma_{35}\sigma_{14}\sigma_{25}\sigma_{13}, &
\sigma_{35}\sigma_{14}\sigma_{25}\sigma_{24}, &
\sigma_{35}\sigma_{25}\sigma_{13}\sigma_{24}. \\
\end{array} \]
In fact, up to cyclic permutation, we have only the following five $5$-braids:
\begin{equation*} 
\begin{array}{ccccc}
\sigma_{13}\sigma_{24}\sigma_{14}\sigma_{25}, &
\sigma_{13}\sigma_{24}\sigma_{35}\sigma_{14}, &
\sigma_{13}\sigma_{24}\sigma_{35}\sigma_{25}, &
\sigma_{13}\sigma_{35}\sigma_{14}\sigma_{25}, &
\sigma_{14}\sigma_{25}\sigma_{24}\sigma_{35}.
\end{array} \end{equation*}
Since $K$ is determined by the conjugacy class of $\beta$, it suffices to consider these.\footnote{In fact, Sage can prove that these five 5-braids all have the same conjugacy class, though we do not need this fact.  We thank Gage Martin for pointing out that this should be the case.}

For each of the five 5-braids $\beta$ in this list, we claim that the closure of the braid $\beta^4$ is not fibered.  Indeed, Sage tells us that the Alexander polynomial of this knot is not even monic:
\begin{verbatim} 
sage: (b**4).alexander_polynomial()                                                                                                                        
5*t^-12 - 9*t^-11 + 8*t^-10 - 4*t^-9 - 3*t^-8 + 12*t^-7 - 17*t^-6 +
12*t^-5 - 3*t^-4 - 4*t^-3 + 8*t^-2 - 9*t^-1 + 5
\end{verbatim}
In other words, we compute for each of these braids that
\begin{multline*}
\Delta_{\widehat{\beta^4}}(t) = 5t^6 - 9t^5 + 8t^4 - 4t^3 - 3t^2 + 12t - 17 \\
+ 12t^{-1} - 3t^{-2} - 4t^{-3} + 8t^{-4} - 9t^{-5} + 5t^{-6},
\end{multline*}
and so Lemma~\ref{lem:fibered-lift} tells us that none of these $\beta$ are exchangeable.
\end{proof}

\begin{remark}
The initial version of this paper on the arXiv used a more involved computational argument which required input from other software, namely SnapPy \cite{snappy}, KLO \cite{KLO}, and Szab\'o's program \cite{szabo-program} for computing knot Floer homology.  We remain grateful to the authors of these programs for making that computation possible, and to the authors of \cite{baker-experiment} for a helpful guide to some of the techniques involved.
\end{remark}

\section{The proof of Theorem \ref{thm:main}}
\label{sec:proof}
In  this section, we prove Theorem \ref{thm:main} following the outline in \S\ref{sec:intro}. Let $\F:=\Z/2\Z$.

\begin{proof}[Proof of Theorem \ref{thm:main}]

Suppose $K\subset S^3$ satisfies the hypotheses of the theorem, 
so that \begin{equation}\label{eqn:Kh-eq}\dim_\F\Khr(K;\F) =5.\end{equation} Then $\dim_\Q\Khr(K;\Q) \leq 5$, by the Universal Coefficient Theorem. This dimension must be odd, and  cannot be 1 or 3 since $K$ would be an unknot or trefoil, respectively, in those cases \cite{km-unknot,bs-trefoil}, which would violate \eqref{eqn:Kh-eq}. Therefore, \[\dim_\Q\Khr(K;\Q) = 5.\] Since $\Khr(K;\F)$ is supported in a single positive $\delta$-grading, the same is true of $\Khr(K;\Q)$. Then, by  \cite[Theorem 1]{bdlls}, which relies on Dowlin's spectral sequence from Khovanov homology to knot Floer homology \cite{dowlin},  the bigraded knot Floer homology of $K$ satisfies
\begin{equation*}
\hfkhat(K;\Q) \cong \hfkhat(T(2,5);\Q).
\end{equation*}
In particular, $K$ has genus 2, and is fibered and strongly quasipositive, as stated in  \cite{bdlls} and in Lemma \ref{lem:hfk-fibered-sqp}, with Alexander polynomial  $\Delta_K(t) = t^2-t+1-t^{-1}+t^{-2}.$

Let us assume for a contradiction that $K\neq T(2,5)$. Then Theorem \ref{thm:main2} implies that there exists a pseudo-Anosov 5-braid $\beta$ with unknotted closure $B$, such that $K$ is the lift of the braid axis $A$ in the branched double cover \[\Sigma(S^3,B) \cong S^3.\]
In particular, $K$ is a doubly-periodic knot with unknotted quotient   $A$ and axis $B$. Below, we prove that $A$ is also braided with respect to $B$.

In \cite[Theorem 1.1]{stoffregen-zhang}, Stoffregen and Zhang use an annular elaboration of Lipshitz and Sarkar's Khovanov spectrum  \cite{lipshitz-sarkar} to prove a Smith-type  inequality relating the Khovanov homology of  $K$ and the annular Khovanov homology of its  quotient $A$, \begin{equation}\label{eqn:AKh}\dim_\F\Kh(K;\F) \geq \dim_\F \AKh(A;\F),\end{equation} where the latter is defined with respect to the embedding  \[A\hookrightarrow S^1\times D^2 = S^3\setminus N(B)\] of $A$ in the solid torus complement of a neighborhood of $B$. Note that \begin{equation*}\label{eqn:Kh-dim}\dim_\F\Kh(K;\F) = 2\dim_\F\Khr(K;\F) = 10.\end{equation*} It then follows from \eqref{eqn:AKh} and  the Universal Coefficient Theorem that \begin{equation}\label{eqn:AKh-ineq}10\geq \dim_\C\AKh(A;\C).\end{equation} In \cite[Theorem~5.16]{xie}, Xie proves that there is a spectral sequence from annular Khovanov homology to annular instanton homology over $\C$ which respects the annular gradings. This implies the inequality  \[\dim_\C\AKh(A,i;\C) \geq \dim_\C\AHI(A,i;\C),\] for each annular grading $i$. 
Since $\lk(A,B)=5$, we must have that  $\AHI(A, k;\C) \neq 0$ for some odd $k\geq 5$, by \cite[Theorem 1.6]{xie-zhang}. It follows that \begin{equation}\label{eqn:AKh-top}\AKh(A, k;\C) \neq 0\end{equation} for some odd $k\geq 5$. Let $k_{\max}$ be the maximal such value of $k$.

In \cite{grigsby-licata-wehrli}, Grigsby, Licata, and Wehrli define an $\mathfrak{sl}_2(\C)$-action on $\AKh(A;\C)$ in which the weight space grading corresponds to the annular grading. In particular, this implies the following unimodality of dimensions, \[ \dim_\C\AKh(A,1;\C) \geq \dim_\C\AKh(A,3;\C)\geq \dim_\C\AKh(A,5;\C) \geq \dots.\] Combined with \eqref{eqn:AKh-ineq} and \eqref{eqn:AKh-top} and the symmetry \[\AKh(A,i;\C)\cong \AKh(A,-i;\C),\] 
this implies that  \[\AKh(A,k_{\max};\C) \cong \C.\]  Indeed, if \[d = \dim_\C \AKh(A,k_{\max};\C),\] then $\dim_\C \AKh(A,i;\C)$ is at least $d$ in each of the $k_{\max}+1 \geq 6$ gradings \[-k_{\max},-k_{\max}+2,\dots,k_{\max}-2,k_{\max},\] and so \[6d \leq d(k_{\max}+1) \leq \dim_\C \AKh(A;\C) \leq 10.\]  Thus, $d=1$ and $5 \leq k_{\max} \leq 9$.
It then follows from \cite[Theorem 1.3]{xie-zhang} that $A$ is a $k_{\max}$-braid with respect to $B$. In particular, we have that $k_{\max}=5$ since $\lk(A,B)=5$.

We have now shown that $A$ and $B$ are mutually braided unknots. In other words, the pseudo-Anosov 5-braid $\beta$ is \emph{exchangeable}, as in \S\ref{sec:computations}. But this contradicts Theorem~\ref{thm:fakeT25}, so we must have $K=T(2,5)$, proving Theorem~\ref{thm:main}.
\end{proof}

\bibliographystyle{alpha}
\bibliography{References}

\end{document}